\documentclass[12pt,a4paper]{article}
\usepackage{amsmath,amssymb}
\usepackage{latexsym}
\usepackage{amsmath}
\usepackage{amssymb}
\usepackage{amsthm}
\usepackage{amscd}
\usepackage{mathrsfs}
\usepackage[all]{xy}
\usepackage{graphicx}
\usepackage{comment}
\usepackage{arydshln}
\usepackage[title]{appendix}

\usepackage{bm}
\usepackage{comment}



\theoremstyle{definition}

\newtheorem{theorem}{Theorem}[section]
\newtheorem{lemma}[theorem]{Lemma}
\newtheorem{corollary}[theorem]{Corollary}
\newtheorem{proposition}[theorem]{Proposition}
\newtheorem{remark}[theorem]{Remark}

\newtheorem{example}[theorem]{Example}

\DeclareMathOperator{\Gal}{Gal}
\DeclareMathOperator{\Hom}{Hom}

\DeclareMathOperator{\Tr}{Tr}

\DeclareMathOperator{\Spec}{Spec}

\usepackage[usenames,dvipsnames]{color}

\begin{document}
\title{Gauss sums and Van der Geer--Van der Vlugt curves}
\author{Daichi Takeuchi 
and Takahiro Tsushima}
\date{}
\maketitle
 \begin{abstract}
We study Van der Geer--Van der Vlugt curves
in a ramification-theoretic view point. 
We give explicit formulae on 
$L$-polynomials of these curves. 
As a result, we show that these curves are supersingular and give sufficient conditions for these curves to 
be maximal or minimal. 
 \end{abstract}
\section{Introduction}

Let $q$ be a power of a prime number $p_0$ and $\mathbb{F}_q$ be a finite field with $q$ elements.  Let $\mathbb{F}$ be an algebraic 
 closure of $\mathbb{F}_q$. 
 Let $\mathrm{Fr}_q \colon 
 \mathbb{F} \to \mathbb{F};\ x \mapsto x^{q^{-1}}$
 be the geometric Frobenius automorphism.  
Let $\ell \nmid p$ be a prime number. 
A smooth projective 
geometrically connected 
curve $C$ over $\mathbb{F}_q$
is said to be supersingular
if all the eigenvalues of 
$\mathrm{Fr}_q$ on $H^1(C_{\mathbb{F}},\overline{\mathbb{Q}}_{\ell})$ 
are $q^{1/2}$ times roots of unity. 
By Tate's theorem, $C$ is 
supersingular if and only if the Jacobian of $C$
is isogenous to a product of supersingular elliptic curves over $\mathbb{F}$.

Let $p$ be a power of $p_0$ and suppose that 
$q$ is a power of $p$.
Let 
$R(x)=\sum_{i=0}^e a_i x^{p^i} 
\in \mathbb{F}_q[x]$
be an additive polynomial of degree $p^e$. 
Let $C_R$ be the affine curve over $\mathbb{F}_q$
defined by 
$y^p-y=x R(x)$ in 
$\mathbb{A}_{\mathbb{F}_q}^2=\Spec \mathbb{F}_q[x,y]$. 
Let $\overline{C}_R$ denote the smooth compactification of $C_R$. 
We call $\overline{C}_R$ the 
Van der Geer--Van der Vlugt curve. 
Assume that 
\[
(p_0,e) \neq (2,0), 
\]
which guarantees that the genus of $\overline{C}_R$ is positive. 
In \cite{GV}, 
Van der Geer and Van der Vlugt 
showed that the family $\{\overline{C}_R\}_R$ has various interesting properties. Among them, they proved that 
they are supersingular 
in the case where $p$ is a prime number. 
This is shown by constructing an algorithm to 
take explicit quotients of $C_R$. The proof 
is complicated in the case where $p$ is even. 
In the case where $p$ is an odd prime number, 
a detailed proof of this theorem is given in \cite{BHMSSV}. 
This theorem is broadly used in Number theory
and Coding theory.  

In this paper, by using tools from $\ell$-adic cohomology theory, we give another method to describe the $L$-polynomials of $\overline{C}_R$ which is simple and can be applied regardless of the parity of the characteristic of $\mathbb{F}_q$. 
 We start with observing that $C_R$ admits an action of a certain Heisenberg group. Using this group action, we decompose the cohomology group $H^1(\overline{C}_{R,\mathbb{F}},\overline{\mathbb{Q}}_\ell)$  into the direct sum of $1$-dimensional representations of $\Gal (\mathbb{F}/\mathbb{F}_q)$. 
Then, applying  Laumon's product formula for epsilon factors (\cite{La}) to each direct summand, we compute the Frobenius eigenvalues in terms of epsilon factors of characters. 
It is classically known that the epsilon factors of characters 
are calculated by Gauss sums. 
As a consequence,  we can show that
$\overline{C}_R$ is supersingular in the case where $p$ is a power of a prime number. 

We describe our results more precisely. Let $E_R(x):=R(x)^{p^e}+\sum_{i=0}^e
(a_i x)^{p^{e-i}}$ and  
\[
H_R:=\{(a,b) \in \mathbb{F}^2 \mid 
E_R(a)=0,\ b^p-b=a R(a)\}. 
\]
This set naturally has a group structure and acts on the curve $C_R$. The center $Z(H_R)$ equals $\{0\}\times\mathbb{F}_p$. 
Let $A_R \subset H_R$ be a maximal 
abelian subgroup.
Then $A_R$ contains the center $\mathbb{F}_p$. 
For a finite abelian group 
$A$, let $A^{\vee}:=\Hom(A,\overline{\mathbb{Q}}^{\times}_{\ell})$ denote the character group. 
For $\psi \in \mathbb{F}_p^{\vee} \setminus \{1\}$, 
let 
\[
A_{\psi}^{\vee}:=\{\xi \in A_R^{\vee} \mid \xi|_{\mathbb{F}_p}=\psi\}. 
\]

Then the $L$-polynomial 
\[
L_{\overline{C}_R/\mathbb{F}_q}(T)
:=\det(1-\mathrm{Fr}_q T;H^1(\overline{C}_{R,\mathbb{F}},\overline{\mathbb{Q}}_{\ell}))
\]
has the following decomposition. 
\begin{theorem}\label{m1}
We assume that 
$A_R \subset \mathbb{F}_q^2$. 
For each $\psi \in \mathbb{F}_p^{\vee}
\setminus \{1\}$ and $\xi \in A_{\psi}^{\vee}$, 
there exists a certain number 
$\tau_{\xi}$ which is $q^{1/2}$ times a root of unity such that a formula 
\[
L_{\overline{C}_R/\mathbb{F}_{q}}
(T)=\prod_{\psi \in \mathbb{F}_p^{\vee} \setminus 
\{1\}} \prod_{\xi \in A_\psi^{\vee}} (1-\tau_{\xi} T)
\] 
holds. Consequently, 
$\overline{C}_R$ is supersingular.
\end{theorem}
We regard $\tau_\xi$ as a Gauss sum attached to $\xi$. An explicit formula for $\tau_\xi$ 
in terms of local epsilon factor is given in Proposition \ref{pc}. In Corollary \ref{4l}, we give another explicit formula for 
$\tau_\xi$ without using epsilon factors. Using 
the Grothendieck trace formula and 
a mechanism of taking quotients of $C_R$ by 
 abelian subgroups of $H_R$, we deduce this formula.  

 A projective smooth geometrically 
 connected curve $C$ over $\mathbb{F}_q$
  is said to be $\mathbb{F}_{q^n}$-maximal (resp.\ $\mathbb{F}_{q^n}$-minimal) if 
 $|C(\mathbb{F}_{q^n})|=q^n+1+2g(C) q^{n/2}$
 (resp.\  $|C(\mathbb{F}_{q^n})|=q^n+1-2g(C) q^{n/2}$), 
 where $g(C)$ denotes the genus of $C$. 
In other words, 
$C$ is $\mathbb{F}_{q^n}$-maximal (resp.\ $\mathbb{F}_{q^n}$-minimal) if and only if 
$\mathrm{Fr}_{q^n}$ acts on 
$H^1(C_{\mathbb{F}},\overline{\mathbb{Q}}_{\ell})$ as 
scalar multiplication by $-q^{n/2}$
(resp.\ $q^{n/2}$). 
Maximal curves are important in Coding theory. 

By evaluating the Gauss
sums $\{\tau_{\xi}\}_{\xi}$, 
Theorem \ref{m1} implies the following. 
\begin{theorem}\label{main}
We assume that 
$A_R \subset \mathbb{F}_q^2$. 
\begin{itemize}
\item[{\rm (1)}] 
The curve $\overline{C}_R$ is 
$\mathbb{F}_{q^{4p_0}}$-minimal.   
\item[{\rm (2)}] 
Assume that $f$ is odd and $p_0 \not\equiv 
1 \pmod 4$. 
Then $\overline{C}_R$ is 
$\mathbb{F}_{q^{2p_0}}$-maximal. 
\item[{\rm (3)}] 
Assume that $p_0=2$ and $H_R \subset \mathbb{F}_q^2$. 
Then $f$ is even and 
$\overline{C}_R$ is $\mathbb{F}_{q^2}$-minimal. 
\end{itemize}
\end{theorem}

This work was supported by JSPS KAKENHI Grant Numbers 20K03529/21H00973 and by 
RIKEN Special Postdoctoral Researcher Program. 
\section{Van der Geer--Van der Vlugt curves}
The curve $C_R$ admits a large automorphism
group containing a Heisenberg group. We recall
this fact briefly. 
Let 
\begin{align*}
E_R(x):&=R(x)^{p^e}+\sum_{i=0}^e (a_ix)^{p^{e-i}}, \\ 
f_R(x,y):&=-\sum_{i=0}^{e-1}\left(\sum_{j=0}^{e-i-1} (a_i x^{p^i} y)^{p^j}
+(x  R(y))^{p^i}\right) \in \mathbb{F}_q[x,y]. 
\end{align*}
We easily check that 
\begin{equation}\label{a}
f_R(x,y)^p-f_R(x,y)=-x^{p^e} E_R(y)+xR(y)+y R(x). 
\end{equation}
Let 
$V_R:=\{x \in \mathbb{F} \mid E_R(x)=0\}$ and 
$H_R:=\{(a,b) \in V_R \times 
\mathbb{F} \mid b^p-b=aR(a)\}$
be the group whose group law is 
defined by 
\[
(a,b) \cdot (a',b')=(a+a',b+b'+f_R(a,a')). 
\]  
We recall some basic properties of $H_R$. 
\begin{lemma}(\cite[Lemma 2.6]{Ts})\label{basic}
\begin{enumerate}
    \item[{\rm (1)}] The center $Z(H_R)$ equals  $\{0\}\times\mathbb{F}_p$. 
    \item[{\rm (2)}] The quotient $H_R/Z(H_R)$ is isomorphic to $V_R$ via $(a,b)\mapsto a$. 
    \item[{\rm (3)}] The mapping $H_R\times H_R\to Z(H_R);\ (x,y)\mapsto xyx^{-1}y^{-1}$ induces a non-degenerate symplectic pairing $\omega_R \colon V_R\times V_R\to\mathbb{F}_p;\ 
    (a,a') \mapsto f_R(a,a')-f_R(a',a)$. 
\end{enumerate}
\end{lemma}
\begin{proof}
    The assertions (1) and (2) are proved in \cite[Lemma 2.6(1)]{Ts}. The assertion (3) is a consequence of \cite[Lemma 2.4]{Ts} and \cite[Lemma 2.6(2)]{Ts}. 
\end{proof}
\begin{lemma}\label{clear}
Let $h \in H_R$. 
The order of $h$ divides $p_0$ if $p_0$ is 
odd and $4$ if $p_0=2$. 
\end{lemma}
\begin{proof}
We write $h=(a,b)$. 
Then $h^i=(i a, ib+\binom{i}{2} f_R(a,a))$ for 
an integer $i \geq 2$. Hence the claim follows. 
\end{proof}
\begin{comment}
Let $g=(a,b),g'=(a',b') \in H_R$. 
We check that  
$[gg'g^{-1}g'^{-1}=(0,f_R(a,a')-f_R(a',a))$.  
Let $(a,b) \in Z(H_R)$. Assume that $a \neq 0$.
Let $W_a:=\{x \in \mathbb{F} \mid f_R(x,a)=f_R(a,x)\}$. 
The condition $a \neq 0$ implies $\dim_{\mathbb{F}_p} W_a=2e-1$. 
By $V_R \subset W_a$, this is a contradiction. Hence $a=0$. 
\end{comment}

The element $H_R \ni (a,b)$ acts on $C_{R,\mathbb{F}}
\ni (x,y)$ by
\begin{equation}\label{std}
(x,y) \cdot (a,b)=(x+a,y+f_R(x,a)+b). 
\end{equation}

Let $A_R \subset H_R$ be a maximal 
abelian subgroup: note that such subgroups are in $1$ to $1$ correspondence with those subgroups of $V_R$ which are maximally totally isotropic with respect to $\omega_R$.
From now on, we always assume that 
\begin{equation}\label{assumption}
(p_0,e) \neq (2,0), \quad 
A_R \subset \mathbb{F}_q^2. 
\end{equation}

There exists an additive polynomial $F_R(x) \in \mathbb{F}_q[x]$ such that 
$F_R \mid E_R$ and 
$A_R$ is the inverse image of $\{x \in \mathbb{F} \mid 
F_R(x)=0\} \subset V_R$ by $H_R \to V_R$.
On the latter condition in \eqref{assumption}, we remark the following. 
\begin{lemma}
Assume that $p \neq 2$. 
The condition $\{x \in \mathbb{F} \mid 
F_R(x)=0\} \subset \mathbb{F}_q$ implies 
that $A_R \subset \mathbb{F}_q^2$. 
\end{lemma}
\begin{proof}
Let $a \in \mathbb{F}_q$ be an element 
satisfying $F_R(a)=0$. 
Then 
$2^{-1}f_R(a,a) \in \mathbb{F}_q$ and 
$(2^{-1}f_R(a,a))^p-2^{-1} f_R(a,a)=a R(a)$. Hence the claim follows.  
\end{proof}
In the characteristic two case, 
the claim in the above lemma does not hold in general 
(cf.\ \cite[Proposition (3.1)]{GV}).  

We consider the finite Galois \'etale morphism  
\[
\phi \colon C_R \to 
\mathbb{A}_{\mathbb{F}_q}^1;\ 
(x,y) \mapsto F_R(x), 
\]
whose Galois group is $A_R$. 
 Let $\psi \in \mathbb{F}_p^{\vee} \setminus \{1\}$
 and 
$A^{\vee}_{\psi}:=\{\xi \in A_R^{\vee} \mid 
\xi|_{\mathbb{F}_p}=\psi\}$. 
For $\xi \in A^{\vee}_{\psi}$, 
let $\mathscr{Q}_{\xi}$ denote 
the smooth sheaf on $\mathbb{A}^1_{\mathbb{F}_q}$ defined by $\xi$ and $\phi$.

Let $\mathscr{L}_{\psi}(s)$ denote the 
Artin--Schreier sheaf on $\mathbb{A}^1_{\mathbb{F}_q}
=\Spec \mathbb{F}_q[s]$ 
defined by $a^p-a=s$ and $\psi \in \mathbb{F}_p^{\vee}$. 
For a morphism of schemes $f \colon X \to \mathbb{A}^1_{\mathbb{F}_q}$, let $\mathscr{L}_{\psi}(f)$
denote the pull-back of 
$\mathscr{L}_{\psi}(s)$ by $f$. 
We write $\mathbb{A}^1$ for the 
affine line over $\mathbb{F}$. 
\begin{lemma}\label{14}
\begin{itemize}
\item[{\rm (1)}]
We have isomorphisms 
\begin{align*}
H_{\rm c}^1(C_{R,\mathbb{F}},\overline{\mathbb{Q}}_{\ell}) 
&\simeq 
\bigoplus_{\psi \in \mathbb{F}_p^{\vee} \setminus \{1\}}H_{\rm c}^1(\mathbb{A}^1,\mathscr{L}_{\psi}(xR(x))), \\
H_{\rm c}^1(\mathbb{A}^1,\mathscr{L}_{\psi}(xR(x)))
& \simeq  
\bigoplus_{\xi \in A^{\vee}_{\psi}} H_{\rm c}^1(\mathbb{A}^1,\mathscr{Q}_{\xi}) \quad 
\textrm{for $\psi \in \mathbb{F}_p^{\vee} \setminus \{1\}.$}
\end{align*}
Moreover, we have $\dim H_{\rm c}^i(\mathbb{A}^1,\mathscr{Q}_{\xi})=0$ if $i\neq1$ and $=1$ if $i=1$. We have $\deg \mathrm{Sw}_\infty(\mathscr{Q}_{\xi})=2$.
\item[{\rm (2)}] 
The canonical map 
$H_{\rm c}^1(C_{R,\mathbb{F}},\overline{\mathbb{Q}}_{\ell}) \to H^1(\overline{C}
_{R,\mathbb{F}},\overline{\mathbb{Q}}_{\ell})$
is an isomorphism. 
\end{itemize}
\end{lemma}
\begin{proof}
We show (1). 
Let 
\[
\phi_1 \colon C_R \to 
\mathbb{A}_{\mathbb{F}_q}^1;\ 
(x,y) \mapsto x, \quad 
\phi_2 \colon \mathbb{A}_{\mathbb{F}_q}^1 \to 
\mathbb{A}_{\mathbb{F}_q}^1;\ x \mapsto F_R(x).
\]
Then $\phi=\phi_2 \circ \phi_1$. 
We have ${\phi_1}_{\ast} \overline{\mathbb{Q}}_{\ell} \simeq \bigoplus_{\psi \in 
\mathbb{F}_p^{\vee}} \mathscr{L}_{\psi}(xR(x))$
and   
${\phi_2}_{\ast} 
\mathscr{L}_{\psi}(xR(x)) 
\simeq \bigoplus_{\xi \in A_{\psi}^{\vee}}\mathscr{Q}_{\xi}$.  
Hence the isomorphisms in (1)
follow. 

We easily check that 
$H_{\rm c}^i(\mathbb{A}^1,\mathscr{L}_{\psi}(xR(x)))=0$ for $i \neq 1$ and 
$\dim H_{\rm c}^1(\mathbb{A}^1,\mathscr{L}_{\psi}(xR(x)))=p^e$ by the Grothendieck--Ogg--Shafarevich formula. 
Hence 
$H_{\rm c}^i(\mathbb{A}^1,\mathscr{Q}_{\xi})=0$
for $i \neq 1$ and 
$\dim H_{\rm c}^1(\mathbb{A}^1,\mathscr{Q}_{\xi})=1$ by $|A_{\psi}^{\vee}|
=p^e$. 
Again by the Grothendieck--Ogg--Shafarevich formula, 
the last claim follows.  

Since the covering $\phi_1\colon C_R\to\mathbb{A}^1_{\mathbb{F}_q}$ is totally ramified, 
 $\overline{C}_R
\setminus C_R$ consists of one point. 
Hence (2) follows. 
\end{proof}

We say that a polynomial $f(x) \in \mathbb{F}_q[x]$ is reduced if 
$\mathbb{F}_q[x]/(f(x))$ is reduced.
Let $\delta(y):=\sum_{i=0}^d b_i y^{p_0^i} \in 
\mathbb{F}_q[y] \setminus \{0\}$ be a reduced  
additive polynomial whose roots 
are contained in $\mathbb{F}_p$. 
We can write $y^p-y=\delta(\nu(y))$ with 
an additive polynomial $\nu(y)$. 
Let $V_{\delta}:=\{y \in \mathbb{F} \mid 
\delta(y)=0\}$. 
We have a surjective homomorphism 
$\mathbb{F}_p \to V_{\delta};\ x \mapsto 
\nu(y)$. This induces  
an injection $V_{\delta}^{\vee} \hookrightarrow 
\mathbb{F}_p^{\vee}$.
Let $C_{\delta}$ denote the smooth 
affine curve over $\mathbb{F}_q$
defined by $\delta(y)=xR(x)$. 
\begin{corollary}\label{Frob}
Let $\tau_\xi$ be the eigenvalue of ${\rm Fr}_q$ on $H^1_c(\mathbb{A}^1,\mathscr{Q}_\xi)$. Then 
we have 
\[
L_{\overline{C}_{\delta}/\mathbb{F}_{q}}(T)
=\prod_{\psi \in V_{\delta}^{\vee}\setminus \{1\}}
\prod_{\xi \in A_{\psi}^{\vee}}
(1-\tau_\xi T). 
\]
\end{corollary}
\begin{proof}
Similarly as Lemma \ref{14}(1), 
we have $H_{\rm c}^1(C_{\delta,\mathbb{F}},\overline{\mathbb{Q}}_{\ell})
\simeq \bigoplus_{\psi \in V_{\delta}^{\vee}\setminus \{1\}} H_{\rm c}^1(\mathbb{A}^1,\mathscr{L}_{\psi}(xR(x)))$. 
We have a sequence of finite morphisms $C_R\to C_\delta\to\mathbb{A}^1_{\mathbb{F}_q}$ whose composite map is given by $\phi_1\colon (x,y)\mapsto x$. This extends to maps between the compactifications $\overline{C}_R\to\overline{C}_\delta\to\mathbb{P}^1_{\mathbb{F}_q}$. Since $\overline{C}_R\setminus C_R$ consists of one $\mathbb{F}_q$-rational point, the same holds true for $C_\delta$. Consequently, the natural map $H_{\rm c}^1(C_{\delta,\mathbb{F}},\overline{\mathbb{Q}}_{\ell}) 
\to 
H^1(\overline{C}_{\delta,\mathbb{F}},\overline{\mathbb{Q}}_{\ell})$ is an isomorphism. 
\end{proof}
We analyze $\tau_\xi$ in the next section. 
\section{Evaluations of $\tau_\xi$}

\subsection{Formulae for $\tau_\xi$ in terms of local epsilon factors}

For a non-archimedean local field $F$, 
let 
$\mathcal{O}_F$ be its ring of integers and 
$\mathfrak{m}_F$ denote its 
maximal ideal. 
Let $U_F^i:=1+\mathfrak{m}_F^i$ for 
$i \geq 1$. 

Let $K:=\mathbb{F}_q((t))$. 
Let $\mathrm{res} \colon \Omega_{K}^1 \to 
\mathbb{F}_q$ denote the residue 
map. 
For a finite extension of fields $E/F$, let $\Tr_{E/F}$ denote the trace map from $E$ to $F$. 
We fix $\psi \in \mathbb{F}_p^{\vee}
\setminus \{1\}$. 
Let 
$\psi_{\mathbb{F}_q}:=\psi \circ 
\Tr_{\mathbb{F}_q/\mathbb{F}_p}$ and 
\[
\Psi \colon K \to 
\overline{\mathbb{Q}}_{\ell}^{\times};\ 
x \mapsto 
\psi_{\mathbb{F}_q}(\mathrm{res}(x d t)). 
\]

 We take a separable closure $\overline{K}$ of 
 $K$ and $W_K$ denote the Weil group of $\overline{K}/K$. 
   Let $j\colon \mathbb{A}^1_{\mathbb{F}_q}\to\mathbb{P}^1_{\mathbb{F}_q}$ be the canonical inclusion and let $x$ be the standard parameter on $\mathbb{A}^1_{\mathbb{F}_q}$. 
We identify the local field of 
$\mathbb{P}_{\mathbb{F}_q}^1$ at $\infty$ with $K$ by setting $t=1/x$. 
Let $\xi \in A_{\psi}^{\vee}$ be an extension of $\psi$. 
Then $\mathscr{Q}_{\xi}$ induces a smooth 
character of $W_K$, which we denote by 
$\xi$. By Lemma \ref{14}(1), we have 
$\deg \mathrm{Sw}(\xi)=2$.

Let $c \in \mathfrak{m}_{K}^{-3} \setminus 
\mathfrak{m}_K^{-2}$.
We consider the Gauss sum
\[
\tau(\xi,\Psi):=q^{-1}\sum_{x \in 
\mathcal{O}_{K}^{\times}/U_{K}^3} \xi^{-1}(cx) \Psi(cx). 
\]
Note that the sum is independent of a choice of $c$.

The second assertion of the following proposition is a consequence of Laumon's product formula. 
\begin{proposition}\label{pc}
\begin{itemize}
    \item[{\rm (1)}] We have $\xi(t)=1$. 
    \item[{\rm (2)}]  We have $\tau_\xi=-\tau(\xi, \Psi). $
\end{itemize}
\end{proposition}
\begin{proof}
(1) Let $F:=\mathbb{F}_q(x)$ and $\mathbb{A}_{F}^\times$ be the idele group of $F$. Let $\tilde{\xi}$ denote the character $\mathbb{A}_F^\times/F^\times\to\overline{\mathbb{Q}}_\ell^\times$ corresponding to $\mathscr{Q}_\xi$. By Artin's reciprocity law, we have the following equality 
\[\tilde{\xi}(x_0)\xi(1/t)=1, 
\]
where $x_0$ denotes the idele in $\mathbb{A}_F^\times$ whose components are given as follows: at the place defined by $0\in\mathbb{A}^1_{\mathbb{F}_q}$, it  is $x$. At the other places, they are $1$. Since the stalk $\mathscr{Q}_{\xi,0}$ at $0$ is isomorphic to $\mathscr{L}_\psi(xR(x))_0$, on which ${\rm Fr}_q$ acts as the identity, we have $\tilde{\xi}(x_0)=1$. The assertion follows. 

 (2) By Lemma \ref{14}(1) and \cite[Th\'eor\`eme (3.2.1.1)]{La}, we have 
\[\mathop{\rm det}(-{\rm Fr}_q; H^1_c(\mathbb{A}^1,\mathscr{Q}_\xi))=q\cdot\varepsilon_\psi(\mathbb{P}^1_{\mathbb{F}_q,(\infty)},j_!\mathscr{Q}_\xi,-dx). 
\]
Let $t=x^{-1}$, which is a uniformizer at $\infty.$ We have
\begin{align*}
\varepsilon_\psi(\mathbb{P}^1_{\mathbb{F}_q,(\infty)},j_!\mathscr{Q}_\xi,-dx)&=\varepsilon_\psi(\mathbb{P}^1_{\mathbb{F}_q,(\infty)},j_!\mathscr{Q}_\xi,t^{-2}dt)\\
&=\xi(t)^{-2}q^{-2}\varepsilon_\psi(\mathbb{P}^1_{\mathbb{F}_q,(\infty)},j_!\mathscr{Q}_\xi,dt). 
\end{align*}
By \cite[(5.8.2)]{Del}, we have 
\[
\varepsilon_\psi(\mathbb{P}^1_{\mathbb{F}_q,(\infty)},j_!\mathscr{Q}_\xi,dt)=q\cdot\tau(\xi,\Psi). 
\]
Then the assertion follows as $\xi(t)=1$ by (1). 
\end{proof}

\subsection{Evaluation of Gauss sums}
Let 
\[
G_{\psi}:=\sum_{x \in \mathbb{F}_{q}} 
\psi_{\mathbb{F}_q}(x^2). 
\]
A similar result to the following lemma 
is found in \cite[Proposition 8.7(ii)]{AS}.
\begin{lemma}\label{pcc}
Assume that $p \neq 2$. 
Let $(\cdot,\cdot)_{K}$ denote the quadratic Hilbert symbol
over $K$. Let $\gamma \in K^{\times}/U_{K}^2$ be the unique element satisfying   
\[
\xi\left(1+x+\frac{x^2}{2}\right)=\Psi(\gamma x) \quad \textrm{for 
$x \in \mathfrak{m}_{K}$}. 
\]
Then we have 
$\tau(\xi,\Psi)=\xi^{-1}(\gamma)\Psi(\gamma) (2 \gamma, t)_{K}G_{\psi}$. 
\end{lemma}
\begin{proof}
The sum $q\tau(\xi,\Psi)=\sum_{x\in\mathcal{O}_K^\times/U_K^3}\xi^{-1}(\gamma  x)\Psi(\gamma x)$ can be computed as follows: 
\begin{align*}
    \sum_{x\in\mathcal{O}_K^\times/U_K^3}\xi^{-1}(\gamma x)\Psi(\gamma  x)&=\sum_{\zeta\in\mathbb{F}_q^\times,z\in\mathfrak{m}_K/\mathfrak{m}^3_K}\xi^{-1}(\gamma\zeta(1+z))\Psi(\gamma\zeta(1+z))\\
    &=\sum_{\zeta\in\mathbb{F}_q^\times}\xi^{-1}(\gamma \zeta)\Psi(\gamma\zeta)\sum_{z\in\mathfrak{m}_K/\mathfrak{m}^3_K}\xi^{-1}(1+z)\xi\left(1+\zeta z+\frac{(\zeta z)^2}{2}\right). 
\end{align*}
The sum $\sum_{z\in\mathfrak{m}_K/\mathfrak{m}^3_K}\xi^{-1}(1+z)\xi(1+\zeta z+\frac{(\zeta z)^2}{2})$ can be written as 
\begin{align*}
&\sum_{y\in\mathfrak{m}_K/\mathfrak{m}^2_K,z\in\mathfrak{m}_K^2/\mathfrak{m}_K^3}\xi^{-1}(1+y)\xi^{-1}(1+z)\xi\left(1+\zeta y+\frac{(\zeta y)^2}{2}\right)\xi(1+\zeta z)\\
    &=\sum_{y\in\mathfrak{m}_K/\mathfrak{m}^2_K}\xi^{-1}(1+y)
    \xi\left(1+\zeta y+\frac{(\zeta y)^2}{2}\right)
    \sum_{z\in\mathfrak{m}_K^2/\mathfrak{m}_K^3}\xi(1+(\zeta-1)z). 
\end{align*}
The last part $\sum_z$ is zero unless $\zeta=1$. 
Therefore, we can compute 
\begin{align*}
    q\tau(\xi,\Psi)&
    =q\xi^{-1}(\gamma)\Psi(\gamma)\sum_{y\in\mathfrak{m}_K/\mathfrak{m}_K^2}\xi^{-1}(1+y)\xi\left(1+y+\frac{y^2}{2}\right)\\
    &=q\xi^{-1}(\gamma)\Psi(\gamma)\sum_{y\in\mathfrak{m}_K/\mathfrak{m}_K^2}
    \Psi\left(\frac{\gamma y^2}{2}\right)=
    q\xi^{-1}(\gamma)\Psi(\gamma)(2\gamma,t)_K G_{\psi}. 
\end{align*}
Hence the assertion follows. 
\end{proof}
The following is shown in an elementary way. 
\begin{lemma}\label{314}
Let $\tau \in \mathbb{Z}[i]$. 
Assume that $|\tau|^2=2^n$ with an integer $n \geq 1$. 
We have 
\[
\frac{\tau}{2^{n/2}} \in 
\begin{cases}
\bigl\{e^{\pm\frac{\pi i}{4}}, 
e^{\pm\frac{3\pi i}{4}}\bigr\} & \textrm{if $n$ is odd}, \\
\{\pm 1, \pm i \} & \textrm{if $n$ is 
even}. 
\end{cases}
\]
\end{lemma}
\begin{proof}
The ring $\mathbb{Z}[i]$ has a unique prime ideal $(1+i)$ which lies over the ideal $(2)$ of $\mathbb{Z}$. Hence we can write $\tau=\zeta(1+i)^n$ where $\zeta=\pm1, \pm i.$ Since $\frac{1+i}{\sqrt{2}}$ is a primitive $8$-th root of unity, the assertion follows. 

\end{proof}
Let $\mu_n:=\{x \in \overline{\mathbb{Q}}_{\ell}
\mid x^n=1\}$ for an integer $n \geq 1$. 
\begin{lemma}\label{ab}
Assume that $p$ is even. 
We have $\tau_{\xi} \in \mathbb{Z}\left[\mu_4\right]$. 
\end{lemma}
\begin{proof}
By Lemma \ref{clear}, 
the image of $\xi$ is contained in 
$\mu_4$. Hence 
the claim 
follows from applying 
Grothendieck's trace formula to 
$\tau_{\xi}=\Tr(\mathrm{Fr}_q; H_{\rm c}^1(\mathbb{A}^1,\mathscr{Q}_{\xi}))$. 
\end{proof}
\begin{corollary}\label{2c}
We write $q=p_0^f$ with a prime number $p_0$. 
\begin{itemize}
\item[{\rm (1)}] 
We have
$\tau_{\xi} \in 
\mu_{4p_0}q^{1/2}. $ 
\item[{\rm (2)}] 
Assume that $f$ is odd and $p_0 \not\equiv 1 \pmod 4$. 
Then we have $\tau(\xi,\Psi)^{2p_0}=-q^{p_0}$. 
\end{itemize}
\end{corollary}
\begin{proof}

Assume that $p_0 \neq 2$. 
Lemma \ref{clear} implies that 
$\xi^{-1}(c) \in 
\mu_{p_0}$. 
Lemma \ref{pcc} implies that $\tau_{\xi}^{2p_0}=(-1)^{f(p_0-1)/2}
q^{p_0}$. Hence the claims (1) and (2) in this case follow. 

Assume that $p_0=2$.
By Proposition \ref{pc}, 
Lemma \ref{314} and Lemma \ref{ab}, 
we know that 
$\tau_{\xi}/q^{1/2} \in \mu_8$ and this value is a primitive $8$-th root of unity if $2 \nmid f$. 
\end{proof}
\subsection{Conclusion}
Now, we give a generalization of 
\cite[Theorems (9.4) and (13.7)]{GV} (cf.\ \cite[Proposition 8.5]{BHMSSV}).
\begin{theorem}\label{mai}
We write $q=p_0^f$ with a prime number $p_0$. 
\begin{itemize}
\item[{\rm (1)}] 
The curve $\overline{C}_R$ is 
$\mathbb{F}_{q^{4p_0}}$-minimal. In particular, 
$\overline{C}_R$ is supersingular.  
\item[{\rm (2)}] 
Assume that $f$ is odd and $p_0 \not\equiv 
1 \pmod 4$. Then $\overline{C}_R$ is 
$\mathbb{F}_{q^{2p_0}}$-maximal. 
\end{itemize}
\end{theorem}
\begin{proof}
The claims follow from Lemma \ref{14},
Proposition \ref{pc} and Corollary \ref{2c}.  
\end{proof}
\begin{example}
We give an example which fits into the situation in 
Theorem \ref{mai}(2). 
Assume that $q=p$ is a prime number satisfying 
$p \not\equiv 1 \pmod 4$. 
Let $R(x)=x^p-x$. We consider the curve 
$\overline{C}_R$. 
We easily check $\mathbb{F}_p \subset V_R$. 
The $\mathbb{F}_p$-subspace 
$\mathbb{F}_p \subset V_R$ 
is totally isotropic. 
Then $A_R=\mathbb{F}_p^2$. Clearly  \eqref{assumption} is satisfied. 
\end{example}
\begin{corollary}\label{mainc}
Assume that $p_0=2$ and 
$H_R \subset \mathbb{F}_q^2$. Write $q=p_0^f$. 
Then $f$ is even and 
$\overline{C}_R$ is $\mathbb{F}_{q^2}$-minimal. 
\end{corollary}
\begin{proof}
Let the notation be as in the proof of 
Corollary \ref{2c}. By Lemma \ref{314}, 
if $f$ is odd, $\tau_{\xi}/q^{1/2}$ is a primitive 
$8$-th root of unity. Hence it suffices to show $\tau_{\xi}/q^{1/2} \in \{\pm 1\}$. 
We assume the contrary. 
Let $\pi_{\psi}$ be the unique irreducible 
representation of $H_R$ whose central character equals  $\psi$. 
Then $\dim \pi_{\psi}=p^e$. 
Hence we have an isomorphism 
$H_{\rm c}^1(\mathbb{A}^1,\mathscr{L}_{\psi}(xR(x))) \simeq \pi_{\psi}$ 
as $H_R$-representations. 
By $H_R \subset \mathbb{F}_q^2$, 
Schur's lemma implies that ${\rm Fr}_q$ acts as a scalar multiplication on $H_{\rm c}^1(\mathbb{A}^1,\mathscr{L}_{\psi}(xR(x))) $: 
\[
a_{\psi}:=\Tr(\mathrm{Fr}_q; H_{\rm c}^1(\mathbb{A}^1,\mathscr{L}_{\psi}(xR(x))))=p^e \tau_{\xi}. 
\]
By the assumption, 
we know that $a_{\psi} \not\in \mathbb{Z}$. 
The Grothendieck trace formula implies that 
\[
\Tr(\mathrm{Fr}_q; H_{\rm c}^1(\mathbb{A}^1,\mathscr{L}_{\psi}(xR(x))))=
-\sum_{x \in \mathbb{F}_q} \psi_{\mathbb{F}_q}(xR(x)).  
\]
Since the image of $\psi$ is 
$\{\pm 1\}$, we have $a_{\psi} \in 
\mathbb{Z}$, which is a contradiction.  
Hence the claims follow. 
\end{proof}
\begin{remark}
Assume that $p_0=2$. In the proof of Lemma \ref{ab}, 
we use only the Grothendieck trace formula. 
Hence we 
obtain Corollary \ref{2c}, Theorem \ref{mai}, and 
Corollary \ref{mainc} without using Laumon's 
product formula in this case. 
\end{remark}
\begin{appendices}  
\section{Another computation of $\tau_{\xi}$}\label{4}
In this appendix, 
without using the local class field theory, 
we directly compute the exact value of $\tau_{\xi}$ in the case $p_0 \neq 2$. 
\subsection{Taking quotients of $C_R$} 
Let $(a,b) \in H_R$ be an element 
such that 
\begin{equation}\label{ad}
a \neq 0, \quad 
b=\frac{f_R(a,a)}{2}.
\end{equation} 
The following lemma is given in \cite[Propositions (9.1) and (13.5)]{GV}, \cite[Proposition 7.2]{BHMSSV} and \cite[Lemma 4.9]{Ts} in the case where $p$ is prime. This lemma gives an algorithm of taking quotients of $C_R$ by certain abelian subgroups in $H_R$. 
\begin{lemma}\label{q0}
Assume $e \geq 1$. 
\begin{itemize}
\item[{\rm (1)}] 
Let 
\[
\Delta_0(x):=-\frac{x}{a}\left(\frac{b}{a}x -
f_R(x,a)\right) 
\] 
and  
\begin{equation}\label{ch}
u:=x^p-a^{p-1}x, \quad 
v:=y-\Delta_0(x). 
\end{equation}
Then there exists an additive polynomial
$R_1(u) \in \mathbb{F}[u]$ 
of degree $p^{e-1}$ satisfying 
$v^p-v=u R_1(u)$. 
The leading coefficient of $R_1(u)$ is 
\[
\begin{cases}
\displaystyle -\frac{a_e}{a^{p-1}} & \textrm{if $e>1$},\\[0.3cm]
\displaystyle -\frac{a_e}{2 a^{p-1}} & \textrm{if $e=1$}.
\end{cases}
\] 
\item[{\rm (2)}] 
Let $U_a:=\{(\xi a,\xi^2 b) \in 
H_R \mid \xi \in \mathbb{F}_p\}$. 
The quotient $C_R/U_a$ is isomorphic to $C_{R_1}$.  
\end{itemize}
\end{lemma}
\begin{proof}
We show the claims in just the same way as \cite[Lemma 4.9]{Ts}. 
\end{proof}
By Lemma \ref{q0}(1), 
\begin{equation}\label{cd}
xR(x)=u R_1(u)+\Delta_0(x)^p-\Delta_0(x). 
\end{equation}
 We write $u(x)$ for $u$. 

Let $(a',b') \in H_R$ 
be an element satisfying \eqref{ad}.
Assume  $\omega_R
(a,a')=0$. Then $(a,b)$ 
 commutes with $(a',b')$ by Lemma
 \ref{basic}(3). Hence the action of $(a',b')$ induces the automorphism of $C_{R_1} \simeq C_R/U_{a}$. 

\begin{lemma}\label{q2}
Let $\pi(a',b'):=(u(a'),2^{-1}f_{R_1}(u(a'),u(a'))) \in \mathbb{F}^2$. 
\begin{itemize}
\item[{\rm (1)}]
We have 
\begin{equation}\label{cd2}
\Delta_0(x+a')+f_{R_1}(u(x),u(a'))
=\Delta_0(x)+\Delta_0(a')+f_R(x,a'). 
\end{equation}
\item[{\rm (2)}] We have $\pi(a',b') \in H_{R_1}$. 
 \item[{\rm (3)}] The action of 
$(a',b')$ on $C_R$ induces 
the automorphism $\pi(a',b')$ on  
$C_{R_1}$.
\end{itemize}
\end{lemma}
\begin{proof}
All the claims are shown in the same way as \cite[Lemma 4.10 and (4.10)]{Ts}. 
\end{proof}
\begin{corollary}\label{acp}
Let the notation be as in Lemma \ref{q0}. 
Let $A \subset V_R$ be a totally isotropic subspace
of dimension $d$ with respect to $\omega_R$. Let $a_1,\ldots,a_d$
be a basis of $A$ over $\mathbb{F}_p$.
Assume $a=a_d$.   
Then $A':=\sum_{i=1}^{d-1} \mathbb{F}_p u(a_i) \subset V_{R_1}$ is a totally isotropic subspace of 
dimension $d-1$ with respect to $\omega_{R_1}$. 
\end{corollary}
\begin{proof}
Assume $\sum_{i=1}^{d-1} x_i u(a_i)=0$
with $x_i \in \mathbb{F}_p$. 
Then $u(\sum_{i=1}^{d-1} x_i a_i)=0$. 
Hence $\sum_{i=1}^{d-1} x_i a_i \in 
\mathbb{F}_p a$. This implies that  
$x_i=0$ for every $1 \leq i \leq d-1$. 
Thus 
$\dim_{\mathbb{F}_p} A'=d-1$. 
Let $1 \leq i\neq j \leq d-1$. 
From \eqref{cd2} and $\omega_R(a_i,a_j)=0$, it follows that $\omega_{R_1}(u(a_i),u(a_j))
 =0$. 
 Thus the claim follows. 
\end{proof}
Let $A\subset V_R$ be a maximal 
totally isotropic subspace with respect to $\omega_R$. 
We identify $A$ with  
an abelian subgroup of $H_R$
by the group homomorphism $A \hookrightarrow H_R;\ 
a \mapsto (a,2^{-1}f_R(a,a))$. 
The following is a generalization of \cite[Theorem 7.4]{BHMSSV}
to the case where $p$ is a power of a prime number. 
\begin{proposition}\label{a5}
Let 
\[
c_A:=
\begin{cases}
\displaystyle (-1)^e \frac{a_e}{2} 
\prod_{\alpha \in A \setminus \{0\}} \alpha^{-1} & \textrm{if $e \geq 1$}, \\
a_0 & \textrm{if $e=0$}. 
\end{cases}
\]
The quotient $C_R/A$
is isomorphic to 
the curve defined by 
$y^p-y=c_{A} x^2$.  
\end{proposition}
\begin{proof}
We note $\dim_{\mathbb{F}_p} A=e$. 
We take a basis $\{a_1,\ldots,a_e\}$ of 
$A$ over $\mathbb{F}_p$. 
Let $a:=a_e$. 
Lemma \ref{q0} implies the finite \'etale morphism 
$C_R \to C_{R_1}$. 
Let $a'_i:=u(a_i) \in 
A'=\sum_{i=1}^{e-1}
\mathbb{F}_p a'_i \subset V_{R_1}$, which is 
totally isotropic with respect 
to $\omega_{R_1}$ by Corollary \ref{acp}.
Taking $a'_{e-1}$ as $a$ and
applying Lemma \ref{q0}, 
we obtain a morphism 
$C_{R_1} \to C_{R_2}$. 
We proceed this process. 
Thus we obtain a finite \'etale morphism 
$\phi \colon C_R \to C_{R_e}$ 
of degree $p^e$.
By Lemma \ref{q0}(1), the curve $C_{R_e}$ is defined by 
$y^p-y=c_A x^2$. By Lemma \ref{q2}(3), 
 the morphism $\phi$ factors through 
 $C_R \to C_R/A \xrightarrow{\phi'} C_{R_e}$. 
Since $C_R \to C_R/A$ has degree $p^e$, $\phi'$ is an isomorphism. 
\end{proof} 
\begin{lemma}
We write $F_R(x)=\sum_{i=0}^e b_ix^{p^i}$. 
Then we have $c_A=(-1)^{e+1}(a_e b_e)/(2 b_0)$ if $e \geq 1$.
\end{lemma}
\begin{proof}
By $\prod_{\alpha \in A \setminus \{0\}} \alpha=-b_0/b_e$, 
the assertion follows from Proposition \ref{a5}. 
\end{proof}
\subsection{Value of $\tau_{\xi}$}
Let $A:=\{x \in \mathbb{F} \mid F_R(x)=0\}$, which is a totally isotropic subspace of $V_R$
with respect to $\omega_R$ in Lemma \ref{basic}(3). 
Assume \eqref{assumption}. In particular, 
\begin{equation}\label{assumption2}
A \subset \mathbb{F}_q \cap V_R.
\end{equation}
Hence there exists an additive polynomial
$a(x) \in \mathbb{F}_q[x]$ such that 
$x^q-x=a(F_R(x))$. We write $q=p^s$.   
For $t \in \mathbb{F}_q$, 
we take $x \in \mathbb{F}$ such that 
$F_R(x)=t$ and let 
\begin{equation}\label{ax}
b(x,t):=\sum_{i=0}^{s-1} (x R(x))^{p^i}-f_R(x,x^q-x). 
\end{equation}
\begin{lemma}
The value $b(x,t)$ is independent of the choice of $x$, for which we write $b(t)$. 
Furthermore, $(a(t),b(t)) \in A_R$. 
\end{lemma}
\begin{proof}
First note that $F_R(a(x))=x^q-x$
 and hence $a(t) \in A$ by $t \in \mathbb{F}_q$. 
Let $y \in A$. Then $\omega_R(a(t),y)=0$, since $A$ is totally isotropic. 
We simply write $\Tr$ for 
$\Tr_{\mathbb{F}_q/\mathbb{F}_p}$. 
We have $\Tr(yR(y))=2^{-1} \Tr(2y R(y))=2^{-1} \Tr(f_R(y,y)^p-f_R(y,y))=0$ 
using \eqref{a} and $y \in \mathbb{F}_q$. 
Let $x':=x+y$. 
By \eqref{assumption2}, we have $y \in \mathbb{F}_q \cap V_R$. 
By $F_R(x)=t$, $x^q-x=a(t)$. 
Using \eqref{a}, we compute 
\begin{align*}
b(x',t)-b(x,t)&=
\sum_{i=0}^{s-1} (y R(x)+x R(y))^{p^i}+\Tr(yR(y))-f_R(y,a(t))\\
&=\sum_{i=0}^{s-1} (f_R(x,y)^p-f_R(x,y))^{p^i}-f_R(y,a(t)) \\
&=f_R(x^q,y)-f_R(x,y)-f_R(y,a(t))=\omega_R(a(t),y)=0. 
\end{align*}
Hence $b(x,t)$ is independent of $x$. 
Again by \eqref{a}, \eqref{assumption2} and $x^q-x=a(t)$, we obtain 
\[
b(t)^p-b(t)=(x R(x))^q-xR(x)-x R(a(t))-a(t) R(x)
%\\&=(x R(x))^q-x^q R(x)-x R(a(t))=x^q R(a(t))-x R(a(t))
=a(t) R(a(t)). 
\]
\end{proof}
\begin{lemma}\label{4.6}
Let $\psi \in \mathbb{F}_p^{\vee} \setminus \{1\}$ and $\xi \in A_{\psi}^{\vee}$. 
Then we have 
\[
\tau_{\xi}=-\sum_{t \in \mathbb{F}_q}
\xi(a(t), b(t)). 
\]
\end{lemma}
\begin{proof}
Let $t \in \mathbb{F}_q$. We take $(x,y) \in C_R$
such that $F_R(x)=t$. 
Recall that $x^q-x=a(t)$. 
Clearly, $y^q-y=\sum_{i=0}^{s-1}(y^p-y)^{p^i}=\sum_{i=0}^{s-1} (x R(x))^{p^i}$. 
Hence $y^q=
y+f_R(x,a(t))+b(t)$. Hence 
$(x^q,y^q)=(x,y) \cdot (a(t),b(t))$ by \eqref{std}. 
By applying the Grothendieck trace formula 
to $\tau_{\xi}=\Tr(\mathrm{Fr}_q; H_{\rm c}^1(\mathbb{A}^1,\mathscr{Q}_{\xi}))$, the assertion follows. 
\end{proof}
By Proposition \ref{a5}, the curve $C_R/A$ is defined by $y^p-y=c_A x^2$. 
We consider the quotient morphism 
\[
\pi \colon C_R \to C_R/A;\ 
(x,y) \mapsto (F_R(x),y-\Delta(x)). 
\]
Then 
\begin{gather}\label{x1}
\begin{aligned}
xR(x)&=c_A F_R(x)^2+\Delta(x)^p-\Delta(x), \\
f_R(x,a)+\frac{f_R(a,a)}{2}&=\Delta(x+a)-\Delta(x) \quad \textrm{for $a \in A$},
\end{aligned}
\end{gather}
where the second equality 
follows from $\pi((x,y) \cdot (a,2^{-1}f_R(a,a)))
=\pi(x,y)$ and \eqref{std}.  
\begin{proposition}\label{4.7}
We have 
\[
b(t)-\frac{f_R(a(t),a(t))}{2}
=\Tr_{\mathbb{F}_q/\mathbb{F}_p}(c_A t^2)
\]
\end{proposition}
\begin{proof}
Using \eqref{x1}, $\Delta(x) \in \mathbb{F}_q[x]$ and $x^q-x=a(t)$, 
we compute
\begin{align*}
b(t)-\frac{f_R(a(t),a(t))}{2}&=
\Tr_{\mathbb{F}_q/\mathbb{F}_p}(c_A t^2)+\Delta(x^q)-\Delta(x)-(\Delta(x+a(t))-\Delta(x)) \\
&=\Tr_{\mathbb{F}_q/\mathbb{F}_p}(c_A t^2). 
\end{align*} 
\end{proof}
We consider the character 
\[
\xi' \colon \mathbb{F}_q \to \overline{\mathbb{Q}}_{\ell}^{\times};\ 
t \mapsto \xi\left(a(t),\frac{f_R(a(t),a(t))}{2}\right). 
\]
Let $\eta \in \mathbb{F}_q^{\times}$ be the element 
such that 
$\xi'(t)=\psi_{\mathbb{F}_q}(\eta t)$ for $t \in \mathbb{F}_q$.  

\begin{corollary}\label{4l}
Let $\psi \in \mathbb{F}_p^{\vee} \setminus \{1\}$ and $\xi \in A_{\psi}^{\vee}$. 
Let $\bigl(\frac{x}{q}\bigr)=x^{\frac{q-1}{2}}$
for $x \in \mathbb{F}_q^{\times}$. 
Then we have 
\[
\tau_{\xi}=-\psi_{\mathbb{F}_q}
 \left(-\frac{\eta^2}{4c_A}\right)
\cdot
\left(\frac{c_A}{q}\right)G_{\psi}. 
\]
\end{corollary}
\begin{proof}
Lemma \ref{4.6} and Proposition \ref{4.7} imply that 
\[
\tau_{\xi}=-\sum_{t \in \mathbb{F}_q}
\psi_{\mathbb{F}_q}(c_A t^2+\eta t)
=-\psi_{\mathbb{F}_q}\left(-\frac{\eta^2}{4c_A}\right)\cdot \left(\frac{c_A}{q}\right)G_{\psi}. 
\]
\end{proof}
\end{appendices}
\begin{comment}
\subsubsection*{Data Availability}
Data sharing not applicable to this article as no datasets were generated or analyzed during the current study.
\subsubsection*{Conflict of interest statement}
We have no conflicts of interest to disclose.
\end{comment}

Daichi Takeuchi\\
RIKEN, Center for Advanced Intelligence Project AIP, Mathematical Science Team, 2-1 Hirosawa, Wako, Saitama, 351-0198, Japan\\
daichi.takeuchi@riken.jp \\

Takahiro Tsushima\\  
Department of Mathematics and Informatics, 
Faculty of Science, Chiba University
1-33 Yayoi-cho, Inage, 
Chiba, 263-8522, Japan \\
tsushima@math.s.chiba-u.ac.jp

\end{document}